\documentclass[12pt,amssymb,verbatim]{amsart}

\usepackage[all]{xy}
\usepackage{amssymb}
\usepackage{amsthm}

\def\cal{\mathcal}

\numberwithin{equation}{section}

\theoremstyle{plain}
\newtheorem{theorem}{Theorem}
\newtheorem*{theorem*}{Theorem}

\newtheorem{lemma}[theorem]{Lemma}

\newtheorem{corollary}[theorem]{Corollary}

\theoremstyle{definition}

\theoremstyle{remark}
\newtheorem{remark}[theorem]{Remark}

\newtheorem{example}[theorem]{Example}

 \def\today{\ifcase\month\or
  January\or February\or March\or April\or May\or June\or
  July\or August\or September\or October\or November\or December\fi
  \space\number\day, \number\year}

\begin{document}
\title[Genus of quaternion algebras]
{Genus of quaternion algebras over \\ formally real fields}
\author[S.V. Tikhonov]{Sergey V. Tikhonov}

\thanks{The research was conducted under the 2026 grant from the President of the Republic of Belarus in the field of science. }

\address{
Belarusian State University, Nezavisimosti Ave., 4,
220030, Minsk, Belarus} \email{tikhonovsv@bsu.by }

\def\cA{{\cal A}}
\def\cB{{\cal B}}
\def\cC{{\cal C}}
\def\cD{{\cal D}}
\def\cV{{\cal V}}
\def\cE{{\cal E}}
\def\cM{{\cal M}}
\def\cR{{\cal R}}
\def\cQ{{\cal Q}}
\def\M{{\rm M}}
\def\cS{{\cal S}}
\def\Symb{\textrm{Symb}}
\def\Gal{\textrm{Gal}}
\def\U{\textrm{U}}
\def\SU{\textrm{SU}}
\def\R{\textrm{R}}
\def\GL{\textrm{GL}}

\def\U{{\rm U}}
\def\SU{{\rm SU}}
\def\SL{{\rm SL}}
\def\op{{\rm op\;}}


\def\cor{\textrm{cor}}
\def\deg{\textrm{deg}}
\def\exp{\textrm{exp}}
\def\Gal{\textrm{Gal}}
\def\ram{\textrm{ram}}
\def\Spec{\textrm{Spec}}
\def\Proj{\textrm{Proj}}
\def\Perm{\textrm{Perm}}
\def\coker{\textrm{coker\,}}
\def\Hom{\textrm{Hom}}
\def\im{\textrm{im\,}}
\def\ind{\textrm{ind}}
\def\int{\textrm{int}}
\def\inv{\textrm{inv}}
\def\min{\textrm{min}}
\def\Br{\textrm{Br}}
\def\char{\textrm{char}}

\def\w{\widehat}

\begin{abstract}

We prove the finiteness of the genus of quaternion division algebras over some formally real fields.
In particular, if $F$ is a formally real field and the second power of the fundamental ideal $I^2(F)$ is torsion-free, then the genus of any quaternion division $F$-algebra is trivial.
For instance, this immediately implies that the genus vanishes for any quaternion division algebra over a formally real pythagorean field.
The proofs are based on methods from the theory of quadratic forms.

\end{abstract}

\maketitle


\def\dd{{\partial}}

\def\into{{\hookrightarrow}}

\def\emptyset{{\varnothing}}

\def\alp{{\alpha}}  \def\bet{{\beta}} \def\gam{{\gamma}}
 \def\del{{\delta}}
\def\eps{{\varepsilon}}
\def\kap{{\kappa}}                   \def\Chi{\text{X}}
\def\lam{{\lambda}}
 \def\sig{{\sigma}}  \def\vphi{{\varphi}} \def\om{{\omega}}
\def\Gam{{\Gamma}}  \def\Del{{\Delta}}  \def\Sig{{\Sigma}}
\def\ups{{\upsilon}}


\def\A{{\mathbb A}}
\def\F{{\mathbb F}}
\def\Q{{\mathbb{Q}}}
\def\CC{{\mathbb{C}}}
\def\PP{{\mathbb P}}
\def\R{{\mathbb R}}
\def\Z{{\mathbb Z}}
\def\X{{\mathbb X}}
\def\R{{\mathbb R}}

\def\Gm{{{\Bbb G}_m}}
\def\Gmk{{{\Bbb G}_{m,k}}}
\def\GmL{{\Bbb G_{{\rm m},L}}}
\def\Ga{{{\Bbb G}_a}}

\def\Fb{{\overline F}}
\def\Hb{{\overline H}}
\def\Kb{{\overline K}}
\def\Lb{{\overline L}}
\def\Yb{{\overline Y}}
\def\Xb{{\overline X}}
\def\Tb{{\overline T}}
\def\Bb{{\overline B}}
\def\Gb{{\overline G}}
\def\Vb{{\overline V}}

\def\kb{{\bar k}}
\def\xb{{\bar x}}

\def\Th{{\hat T}}
\def\Bh{{\hat B}}
\def\Gh{{\hat G}}

\def\Xt{{\tilde X}}
\def\Gt{{\tilde G}}

\def\gm{{\mathfrak m}}
\def\gp{{\mathfrak p}}
\def\gq{{\mathfrak q}}

\def\textrm#1{\text{\rm #1}}

\def\res{\textrm{res}}
\def\cor{\textrm{cor}}
\def\ed{\textrm{ed}}

\def\tors{_{\textrm{tors}}}      \def\tor{^{\textrm{tor}}}
\def\red{^{\textrm{red}}}         \def\nt{^{\textrm{ssu}}}
\def\sc{^{\textrm{sc}}}
\def\sss{^{\textrm{ss}}}          \def\uu{^{\textrm{u}}}
\def\ad{^{\textrm{ad}}}           \def\mm{^{\textrm{m}}}
\def\tm{^\times}                  \def\mult{^{\textrm{mult}}}
\def\tt{^{\textrm{t}}}
\def\uss{^{\textrm{ssu}}}         \def\ssu{^{\textrm{ssu}}}
\def\cf{^{\textrm{cf}}}
\def\ab{_{\textrm{ab}}}

\def\et{_{\textrm{\'et}}}
\def\nr{_{\textrm{nr}}}

\def\SB{\textrm{SB}}

\def\<{\langle}
\def\>{\rangle}

\def\til{\;\widetilde{}\;}


\font\cyr=wncyr10 scaled \magstep1%
\def\Bcyr{\text{\cyr B}}
\def\Sh{\text{\cyr Sh}}
\def\Ch{\text{\cyr Ch}}

\def\lpsi{{{}_\psi}}
\def\bks{{\backslash}}

\def\Br{\textrm{Br}}
\def\Pic{\textrm{Pic}}
\def\Bt{{{}_2\textrm{Br}}}
\def\Bn{{{}_n\textrm{Br}}}
\def\Grass{\textrm{Grass}}




\section{Introduction}

Let $F$ be a field and $\Br(F)$ its Brauer group.
The genus $\mathbf{gen}(\cD)$ of a finite-dimensional central division algebra $\cD$ over a field $F$ is defined as the set of classes $[\cD']\in \Br(F)$,
where $\cD'$ is a central division $F$-algebra having the same maximal subfields as $\cD$.
Various modifications of the concept of genus are discussed in \cite{ChRaRa15} and \cite{KrMaRaRoSa22}.
Results concerning the genus of algebras and algebraic groups are presented in the surveys \cite{RaRa20} and
\cite{Ra24}.

When $F$ is a global number field, the Albert-Brauer-Hasse-Noether theorem (\cite[\S 18.4]{Pi82}) guarantees that the genus of any central division $F$-algebra is finite. Moreover, the genus of any quaternion division algebra over $F$ is trivial (i.e., $|\mathbf{gen}(\cD)|=1$) (see \cite[Proposition 3.1]{ChRaRa15}).

A systematic framework for establishing the finiteness of $\mathbf{gen}(\cD)$ and bounding its cardinality was introduced in \cite{ChRaRa13}; this method relies on analyzing the unramified Brauer group relative to a suitable family of discrete valuations on the field. Subsequently, it was demonstrated in \cite{ChRaRa16} and \cite{ChRaRa20} that every central division algebra over a finitely generated field has a finite genus.
The question remains open whether there is a quaternion division algebra over a finitely generated field of characteristic $\ne 2$ having nontrivial genus (\cite[Question 8.8]{RaRa20}).

The behavior of the genus for division algebras over discrete valued fields was investigated in \cite{Sr25}. Specifically, for the case of quaternion algebras, the genus was shown to be trivial over a wide class of fields, including higher local fields, function fields of curves defined over higher local fields, and function fields of curves over real closed fields.

According to \cite{RaRa10}, if the genus of every central quaternion division algebra $\cD$ over a field $F$ ($\char(F) \ne 2$) is trivial, this property passes to central quaternion division algebras over the rational function field $F(x)$. This finding was later extended to central division algebras of exponent 2 in \cite{ChRaRa13}. However, it is not yet known whether there is
a quaternion division algebra over the field of rational functions
$F(x)$ over some field $F$ having nontrivial genus (\cite[Question 8.6]{ChRaRa24}).

For arbitrary fields, the genus of an algebra may well be infinite. In \cite{GaSa10}, an example of a quaternion algebra $\cD$ over a large center was constructed via an iterative composition of function fields, yielding a nontrivial genus. Furthermore, \cite{Me14} established the existence of quaternion division algebras with infinite genus over fields of infinite transcendence degree over their prime subfields. These results were later generalized to division algebras of an arbitrary prime degree in \cite{Ti16}. More recently, in \cite[Theorem 5.2]{KrLi24}, the authors provided examples of infinite families of division algebras over a common field that share the same finite splitting fields, as well as splitting fields arising as function fields of curves of genus at most 1.
In \cite{Ti25}, the genus of finite-dimensional division algebras was shown to be finite over a broad class of infinitely generated fields.

In this paper, we prove the finiteness of the genus of quaternion division algebras over some formally real fields.
The proofs are based on methods from the theory of quadratic forms, in particular employing Pfister's local-global principle for quadratic forms over formally real fields.
The main results are the following. The necessary notation is given below.

\begin{theorem*}[Theorem \ref{th: I2(F)}]
Let $F$ be a formally real field such that $I_t^2(F)$ is finite. Let $n := |I_t^2(F)|$.
Let $\cA=(a,b)$ be a quaternion division $F$-algebra. Then the number of quaternion division $F$-algebras
$\cB=(c,d)$ such that $\cA$ and $\cB$ satisfy condition~($*$) does not exceed $n$.
In particular, $|\mathbf{gen}(\cA)| \le n$.
\end{theorem*}

As a corollary, we obtain that the genus of any quaternion division algebra over a formally real pythagorean field is trivial.

\begin{theorem*}[Theorem \ref{th: variety}]
Let $F$ be a formally real field which is a finitely generated extension of some subfield $k$.
Assume that the torsion of $I_{nr}^2(F/k)$ is finite. Let $n$ be the number of torsion elements in $I_{nr}^2(F/k)$. Then
$$
|\mathbf{gen}(\cA)| \le n
$$
for any quaternion division $F$-algebra $\cA$.
\end{theorem*}


\section{Notation and Preliminary Results}

We first recall some basic notions from the theory of quadratic forms and quaternion algebras.
For a field $F$ of characteristic not 2 and non-zero elements $a, b \in F$, $(a, b)_F$ denotes the corresponding quaternion algebra.
We omit the subscript $F$ if the base field is clear from the context.
The norm form of a quaternion $F$-algebra $\mathcal{A}=(a, b)_F$ is the 2-fold Pfister form $n_{\mathcal{A}} :=\langle \langle a, b \rangle \rangle = \langle 1, -a, -b, ab \rangle$. Quaternion $F$-algebras $\mathcal{A}$ and $\mathcal{B}$ are isomorphic if and only if their corresponding norm forms $n_{\mathcal{A}}$ and $n_{\mathcal{B}}$ are isometric. The pure subform of $n_{\mathcal{A}}$, denoted by $n_{\mathcal{A}}'$, is the 3-dimensional form $\langle -a, -b, ab \rangle$.
For a quadratic form $\phi$ over $F$, $D_F(\phi)=D(\phi)$ denotes the set of values $d\in F^*$ represented by $\phi$.
A quadratic extension $F(\sqrt d)$ for $d \in F^* \setminus {F^*}^2$
splits $\mathcal{A}=(a,b)$ if and only if $d \in D(-n_{\mathcal{A}}')$.
For a field extension $K/F$ and a central simple $F$-algebra $\mathcal{A}$, $\mathcal{A}_K$ denotes the tensor product $\mathcal{A}\otimes_F K$.
We will write $\cA \sim \cB$ to refer to the fact
that the central simple algebras $\cA$ and $\cB$ are similar.

We say that quaternion division $F$-algebras $\mathcal{A}=(a,b)$ and $\mathcal{B}=(c,d)$ satisfy condition~$(*)$ if:
\begin{equation}
\label{cond:star}
\tag{$*$}
\begin{minipage}[t]{0.85\textwidth}
the algebra $\mathcal{A}=(a,b)$ is split by the fields $F(\sqrt{c})$ and $F(\sqrt{d})$, and
the algebra $\mathcal{B}=(c,d)$ is split by the fields $F(\sqrt{a})$ and $F(\sqrt{b})$.
\end{minipage}
\end{equation}
Note that if $\mathbf{gen}(\mathcal{A}) = \mathbf{gen}(\mathcal{B})$, then the algebras $\mathcal{A}$ and $\mathcal{B}$ clearly satisfy condition~$(*)$.

Let $F$ be a formally real field (i.e., $-1$ is not a sum of squares in $F$).
We denote by $W(F)$ the Witt ring of $F$. Then $I^n(F)$ denotes the $n$-th power of the fundamental ideal $I(F)$.
Moreover, $W_t(F)$ is the additive torsion in $W(F)$ and $I_t^n(F) := I^n(F) \cap W_t(F)$.
For a formally real field $F$, let $X(F)$ be the space of orderings. For $P \in X(F)$, $F_P$ denotes the real closure of $F$ at $P$.
If a field $F$ is real closed, there is, up to an $F$-algebra isomorphism, only one quaternion division $F$-algebra, which is the algebra $(-1,-1)_F$.

Pfister's local--global principle (see, e.g., \cite[Theorem 31.22]{ElKaMe08}) states that the sequence
$$
0 \longrightarrow W_t(F) \longrightarrow W(F) \longrightarrow \prod_{P \in X(F)}  W(F_P)
$$
is exact.

A pythagorean field is a field in which every sum of two squares is a square.
A formally real field is called hereditarily pythagorean if every formally real algebraic extension of it is pythagorean.
A euclidean field is an ordered field for which every non-negative element is a square.
In particular, any euclidean field is pythagorean.
A formally real field $F$ is called hereditarily euclidean if every finite real field extension of $F$ is euclidean.
A field that is both hereditarily pythagorean and euclidean is already hereditarily euclidean.
A field $F$ satisfies property $A_2$ if there are no anisotropic $2$-fold Pfister forms of finite order in $W(F)$.
Recall that the Pythagoras number $p(F)$ of a field $F$ is the smallest integer $n$ such that every element that can be expressed as a sum of
squares in $F$ can be expressed as a sum of at most $n$ squares.

The following lemma will be needed.

%

\begin{lemma} \label{lem:D(n)}
Let $F$ be a formally real field for which $I^2(F)$ is torsion-free.
For two $2$-fold Pfister forms $\sigma$ and $\phi$ over $F$, the following statements are
equivalent:
$$
(1) \; \sigma \cong \phi,   \;\;\;  (2) \; D(\sigma')=D (\phi'),  \;\;\;  (3)  \; D(\sigma)= D(\phi).
$$
\end{lemma}

\begin{proof}
This is a particular case of \cite[Proposition 2.2]{ElLa72}.
\end{proof}

\section{Fields with Finite Torsion in $I^2(F)$}

We start with the following

\begin{lemma}  \label{lem: FP}
Let $F$ be a formally real field, and let $\cA=(a,b)$, $\cB=(c,d)$ be quaternion division $F$-algebras satisfying condition~($*$). Then
$$
\cA_{F_P}\cong \cB_{F_P}
$$
for any $P \in  X(F)$.
\end{lemma}

\begin{proof}
Let $P \in X(F)$. Assume that ${\cA}_{F_P}\cong (-1,-1)$. Then $a$ and $b$ are negative with respect to $P$.
Since $F(\sqrt{c})$ splits the algebra $\cA$,
$$
c=ax_1^2+bx_2^2-abx_3^2
$$
for some $x_1, x_2, x_3 \in F$. Then $c$ is negative with respect to $P$.
Analogously, $d$ is negative with respect to $P$.
Hence  ${\cB}_{F_P}\cong (-1,-1)$ and ${\cA}_{F_P}\cong {\cB}_{F_P}$.

Similarly, if ${\cB}_{F_P}\cong (-1,-1)$, then ${\cA}_{F_P}\cong (-1,-1)$.
\end{proof}

\begin{theorem}  \label{th: I2(F)}
Let $F$ be a formally real field such that $I_t^2(F)$ is finite. Let $n := |I_t^2(F)|$.
Let $\cA=(a,b)$ be a quaternion division $F$-algebra. Then the number of quaternion division $F$-algebras
$\cB=(c,d)$ such that $\cA$ and $\cB$ satisfy condition~($*$) does not exceed $n$.
In particular, $|\mathbf{gen}(\cA)| \le n$.
\end{theorem}

\begin{proof}
Let $\cB$ be a quaternion division $F$-algebra such that $\cA$ and $\cB$ satisfy condition~($*$). Then $\cB \sim \cA \otimes_F \cC$ for
some quaternion $F$-algebra $\cC$.
By Lemma \ref{lem: FP}, ${\cA}_{F_P}\cong {\cB}_{F_P}$ for any $P \in X(F)$. Hence $\cC_{F_P}$ splits for any $P \in X(F)$.
Consequently, $n_{\mathcal{\cC}_{F_P}}$ is hyperbolic for any $P \in X(F)$.

Then by Pfister's local-global principle, the corresponding norm form
$n_{\cC}$ belongs to $I_t^2(F)$. Since, up to isometry,
there are at most $n$ torsion 2-fold Pfister forms over $F$,
the number of quaternion $F$-algebras split over $F_P$ for any $P \in X(F)$ does not exceed $n$.
Then the number of quaternion division $F$-algebras
$\cB$ such that $\cA$ and $\cB$ satisfy condition~($*$) does not exceed $n$.
\end{proof}

From the theorem we immediately obtain the following consequence.

\begin{corollary}  \label{cor: I2(F)}
Let $F$ be a formally real field such that the ideal $I^2(F)$ is torsion-free. Then the genus is trivial
for any quaternion division $F$-algebra.
\end{corollary}

We also present another proof of this corollary, which, in our opinion, is of independent interest.

\begin{proof}
Let $\cA$ and $\cB$ be quaternion division $F$-algebras.
Note that
$$
\mathbf{gen}(\cA)=\mathbf{gen}(\cB) \Leftrightarrow D(-n_{\cA}')=D(-n_{\cB}') \Leftrightarrow D(n_{\cA}')=D(n_{\cB}').
$$

If $\mathbf{gen}(\cA)=\mathbf{gen}(\cB)$, by Lemma \ref{lem:D(n)}, $n_{\cA} \cong n_{\cB}$. Hence $\cA \cong \cB$.
\end{proof}


\begin{remark} \label{rem: Classification Theorem}
The proof of Corollary \ref{cor: I2(F)} also can be obtained from Classification Theorem 2 in \cite{ElLa74}
which contains the characterization of fields with torsion-free $I^2(F)$.
The following statements are equivalent:

(1) Quadratic forms over $F$ are  classified by ``dim'', ``det'' and the total signature.

(2) $I^2(F)$ is torsion-free.

(3) If an $F$-quaternion algebra splits in every real closure of $F$, then it splits over $F$.

(4) If an $F$-quaternion algebra splits over some $F(\sqrt{w})$ where $w$ is totally positive,
then it splits over $F$.

(5) All binary forms $\langle 1, a \rangle$ represent all totally positive elements of $F$.

\end{remark}

\begin{remark} \label{rem: not real}
If the field $F$ is not formally real and $I^2(F)$ is torsion-free, then $I^2(F)=\{0\}$ because $W(F)$ is a torsion group. Consequently, there are no quaternion division algebras over $F$.
\end{remark}

Below, we present some fields over which the genus of any quaternion division algebra is trivial.

\begin{corollary} \label{cor: examples}
If $F$ is one of the following fields, then the genus of any quaternion division $F$-algebra is trivial

(1) a formally real pythagorean field;

(2) a formally real field such that $\Br(F)$ has no elements of order 4;

(3) a formally real field extension of transcendence degree $1$ over
a hereditarily pythagorean euclidean field;

(4) a formally real field $F$ satisfying property $A_2$.

\end{corollary}

\begin{proof}
(1)
A formally real field $F$ is pythagorean if and only if its Witt ring $W(F)$ is torsion-free (see \cite[Corollary 31.9]{ElKaMe08}).


(2) By \cite[Proposition 3.3]{BeLe01}, $I^2(F)$ is torsion-free. Note that in \cite{BeLe01} other notation is used, but by Pfister's local-global principle, they are equivalent.

(3) Since $F$ is euclidean,  $|F^*/{F^*}^2|=2$. Then by \cite[Theorem 6.6]{BeDaGrMaZa24}, $I^2(F)$ is torsion-free.

(4) By \cite[Section 35.A]{ElKaMe08}, $I^2(F)$ is torsion-free for a formally real field $F$ satisfying property $A_2$.
\end{proof}

\begin{example}
Let $F$ be a function field in one variable over any real closed field. Assume that $F$ is formally real. Then the genus of any quaternion division $F$-algebra is trivial. Indeed,
by Tsen's theorem (\cite[Theorem 6.2.8]{GiSz06}), any central simple $F$-algebra splits over the field $F(\sqrt{-1})$. Hence $\Br(F)$ has no elements of order 4.

The same result can be obtained using  \cite[Corollary 35.29]{ElKaMe08} which says that if $F$ is a finitely generated extension of transcendence degree $n$ over a real closed field,
then $I^{n+1}(F)$ is torsion-free.

The general case of a function field in one variable over any real closed field, including a non-real case, is considered in \cite[Corollary 3.7 (4)]{Sr25}.
\end{example}

\begin{corollary}  \label{cor: Pyth}
Let $F$ be a formally real field
and let $\cA=(a,b)$, $\cB=(c,d)$ be quaternion division $F$-algebras satisfying condition~($*$). Then
$$
\cA_{K}\cong \cB_{K}
$$
for any formally real field $K$ containing $F$ with torsion-free $I^2(K)$.
\end{corollary}

\begin{proof}
As in the proof of Theorem \ref{th: I2(F)}, we obtain that $\cB \sim \cA \otimes_F \cC$ for
some quaternion $F$-algebra $\cC$ and
$n_{\cC}$ belongs to $I_t^2(F)$. Then $n_{\cC_K}$ belongs to $I_t^2(K)$. Since $I^2(K)$ is torsion-free, $n_{\cC_K}$ is hyperbolic.
Hence $\cC_K$ is a matrix algebra. Thus $\cA_{K}\cong \cB_{K}$.
\end{proof}

\begin{corollary}  \label{cor: I3(F)}
Let $F$ be a formally real field such that the ideal $I^3(F)$ is torsion-free.
Let $\cA$ and $\cB$ be quaternion division $F$-algebras such that
$\mathbf{gen}(\cA) = \mathbf{gen}(\cB)$. Then $[\cA \otimes_F \cB] \in 2 \Br(F)$.
\end{corollary}

\begin{proof}
As in the proof of Theorem \ref{th: I2(F)}, we conclude that $\cB \otimes_F \cA \sim \cC$ for
some quaternion $F$-algebra $\cC$ and
$n_{\cC}$ belongs to $I_t^2(F)$. Then
$$
2 n_{\cC} = \langle 1,1 \rangle \otimes n_{\cC} \in I_t^3(F)
$$
and the order of $n_{\cC}$ does not exceed $2$. In view of \cite[Corollary 5.14]{LaLeTi93},
$[\cC] \in 2\Br(F)$.
\end{proof}

\begin{remark}
Equivalent properties of quaternion algebras from $2 \Br(F)$ are presented in \cite[Corollary 5.14]{LaLeTi93}.
\end{remark}


\begin{corollary}  \label{cor: -1}
Let $F$ be a formally real field with $p(F)\le 2$.
Let $\cA$ be a quaternion division $F$-algebra such that $F(\sqrt{-1})$ splits $\cA$.
Then $\mathbf{gen}(\cA)$ is trivial.
\end{corollary}

\begin{proof}
Assume that $\cB$ is a quaternion division $F$-algebra and
$\mathbf{gen}(\cA) = \mathbf{gen}(\cB)$.
As in the proof of Theorem \ref{th: I2(F)}, we obtain that $\cB \otimes_F \cA \sim \cC$ for
some quaternion $F$-algebra $\cC$ such that
$\cC_{F_P}$ splits for any $P \in X(F)$.

Since $F(\sqrt{-1})$ splits $\cB$, $\cC$ is also split by $F(\sqrt{-1})$. Hence
$$
\cC \cong (-1,c)
$$
for some totally positive element $c\in F$.
By the Artin-Schreier theorem, $c$ is a sum of squares in $F$. Hence, $c$ is a sum of at most two squares, which implies that $\cC$ is a matrix algebra. Thus $\cA \cong \cB$.
\end{proof}

\section{Function Fields of Varieties}

Before formulating the main result of this section, we recall some notions.

Let $K$ be a field of characteristic $\ne$ 2. For a discrete valuation $v$ of $K$, $\overline{K}_v$ denotes its residue field.
The $2$-torsion of the Brauer group $\Br(K)$ is denoted by ${_2 \Br}(K)$. There is the residue map
$$
\rho_v : {_2 \Br}(K) \longrightarrow \textrm{Hom} (G_v, \Z/2\Z),
$$
where $G_v$ is the absolute Galois group of  $\overline{K}_v$ (see, e.g., \cite[\S 10]{Sa99}).
Note that for a quaternion $K$-algebra $(a, b)$,
$$
\rho_v([(a, b)]) = (-1)^{v(a)v(b)}\overline{a^{v(b)}b^{v(a)}} \in \overline{K}_v^*/{\overline{K}_v^*}^2 \cong \textrm{Hom} (G_v, \Z/2\Z)
$$
by the tame symbol formula. A quaternion $K$-algebra $\cA$ is said to be unramified at $v$, if $\rho([\cA])=1$.

Let $K$ be a finitely generated extension over some subfield $k$,
$V_K$ denotes the set of all rank one discrete valuations on $K$, which are trivial
on $k$. If $v\in V_K$, $\pi_v$ is a fixed uniformizer for $v$ and $k(v)$ is the residue field at $v$, then the corresponding second residue homomorphism
$$
\partial_{\pi_v} : W(K) \longrightarrow W(k(v))
$$
is defined as follows (see, e.g., \cite[Section 19.A]{ElKaMe08}). For $a=u \pi_v^i$, where $i$ is an integer and $u$ is a unit in the corresponding valuation ring,
$$
\partial_{\pi_v} (\langle a \rangle) = \left\{
\begin{array}{cl}
0 & i \mbox{ is even,}\\
\langle \overline{u} \rangle &  i \mbox{ is odd.}\\
\end{array}
\right.
$$
The unramified Witt group $W_{nr}(K/k)$ is defined as the kernel of the group homomorphism
$$
W(K) \stackrel{\oplus \partial_{\pi_v}}\longrightarrow  \bigoplus_{v \in V_K} W(k(v))
$$
(see \cite{CoSu98}). Let
$$
I_{nr}^2(K/k) := W_{nr}(K/k) \cap I^2(K).
$$

Let $X$ be a smooth geometrically integral variety over a field $k$ of characteristic $\ne$ 2 and let $k(X)$ denote its function field. Let also $X^{(1)}$ denote the set of codimension $1$ points of $X$.
Since $X$ is smooth, any element $x$ of $X^{(1)}$ defines a rank one discrete valuation $v_x$ on $k(X)$. Let $k(x)$ be the residue field of $v_x$.
If $x\in X^{(1)}$ and $\pi_x$ is a fixed uniformizer for $v_x$, then the corresponding second residue homomorphism is denoted by $\partial_{\pi_x}$. It is clear that
$$
W_{nr}(k(X)/k) \subset \textrm{ker} \left(W(k(X)) \stackrel{\oplus \partial_{\pi_x}}\longrightarrow  \bigoplus_{x \in X^{(1)}}W(k(x)) \right).
$$

\begin{theorem}  \label{th: variety}
Let $F$ be a formally real field which is a finitely generated extension over some subfield $k$.
Assume that the torsion of $I_{nr}^2(F/k)$ is finite. Let $n$ be the number of torsion elements in $I_{nr}^2(F/k)$. Then
$$
|\mathbf{gen}(\cA)| \le n
$$
for any quaternion division $F$-algebra $\cA$.
\end{theorem}

\begin{proof}
Assume that $\cA$ and $\cB$ are quaternion division $F$-algebras and $[\cB] \in \mathbf{gen}(\cA)$.
As in the proof of Theorem \ref{th: I2(F)} we obtain that $\cB \sim \cA \otimes_F \cC$ for
some quaternion $F$-algebra $\cC$ and
$n_{\cC}$ belongs to $I_t^2(F)$.

Let $v \in V_F$. By \cite[Lemma 2.5]{ChRaRa13}, $\rho_{v}([\cA])= \rho_{v}([\cB])$. Hence
$$
\rho_{v}([\cC]) = \rho_{v}([\cA])  \rho_{v}([\cB]) =1
$$
and the algebra $\cC$ is unramified at $v$. Using tame symbol formula for $\rho_{v}$ and the definition of
$\partial_{\pi_v}$ one can show that $\partial_{\pi_v}(n_{\cC})=0$.
Then $n_{\cC}$ belongs to $W_{nr}(F/k)$ and hence to the torsion of $I_{nr}^2(F/k)$.
The number of quaternion division $F$-algebras
$\cC$ such that $n_{\cC}$ belongs to the torsion of $I_{nr}^2(F)$ does not exceed $n$.
Thus $|\mathbf{gen}(\cA)| \le n$.
\end{proof}

Below, we consider examples of fields satisfying the conditions of the theorem above.

\begin{example}  \label{example: surface}
Let $X$ be a smooth projective geometrically integral real surface such that $X(\R)\ne \varnothing$.
Let $W(X)$ denote the Witt group of $X$ (\cite[Chapter I]{Kn77}.
By \cite[Corollary 3.2]{Fe87}, $W(X)$ is isomorphic to the kernel of the homomorphism
$$
W(\R(X)) \stackrel{\oplus \partial_{\pi_x}}\longrightarrow  \bigoplus_{x \in X^{(1)}}W(\R(x)).
$$
Then $W_{nr}(\R(X)/\R) \subset W(X)$.

In view of \cite[Theorem 3.1]{Su90} (see also \cite[Theorem 48]{Ba05}), the torsion of $W(X)$ is finite. Hence, the torsion of
$I_{nr}^2(\R(X)/\R)$ is also finite. Thus any quaternion division $\R(X)$-algebra has finite genus by Theorem \ref{th: variety}.

Finally, we note that the field $\R(X)$ satisfies the conditions of the Corollary  \ref{cor: I3(F)}  since by \cite[Corollary 35.29]{ElKaMe08},  $I^3(\R(X))$ is torsion-free.
\end{example}

\begin{example}  \label{example: hyperbolic affine sphere}
Let $k$ be a formally real pythagorean field and $n=1$ or $n=2$.
Consider the hyperbolic affine $(2n-1)$-sphere $\Sigma^{2n-1}$ defined over $k$ by the equation
$$
\sum_{i=1}^{n}T_{i}S_{i}=1.
$$

Although the rationality of $\Sigma^{2n-1}$ together with Stability Theorem \cite[Theorem A]{RaRa10} and Corollary \ref{cor: examples} already yield that the genus
of any quaternion division $k(\Sigma^{2n-1})$-algebra is trivial, we establish this result independently via Theorem \ref{th: variety}.

In view of \cite[Corollary 92 and Corollary 95]{Ba05} (see also \cite{BaWa02}), the Witt group $W(\Sigma^{2n-1})$ of $\Sigma^{2n-1}$ is isomorphic to the kernel of the homomorphism
$$
W(k(\Sigma^{2n-1})) \stackrel{\oplus \partial_{\pi_x}}\longrightarrow  \bigoplus_{x \in {\Sigma^{2n-1}}^{(1)}}W(k(x)).
$$
Moreover, by \cite[Theorem 101]{Ba05},
$$
W(\Sigma^{2n-1}) = \left\{
\begin{array}{cl}
W(k) \oplus W(k) & \mbox{ if $n\equiv 1$ modulo 4,}\\
W(k) &  \mbox{ if $n\not\equiv 1$ modulo 4.}\\
\end{array}
\right.
$$

Since $k$ is formally real pythagorean, $W(k)$ is torsion-free by \cite[Corollary 31.9]{ElKaMe08}. Then $W(\Sigma^{2n-1})$ is torsion-free.
This implies that $I_{nr}^2(k(\Sigma^{2n-1})/k)$ is torsion-free. In view of Theorem \ref{th: variety}, the genus of any quaternion division $k(\Sigma^{2n-1})$-algebra is trivial.
\end{example}




\begin{thebibliography}{99}

\frenchspacing


\bibitem{Ba05}
P. Balmer, {\em Witt Groups}, in: Handbook of K-theory, Vol. 2, Springer-Verlag, Berlin-Heidelberg, 2005, pp. 539--576.

\bibitem{BaWa02}
P. Balmer and C. Walter, {\em A Gersten--Witt spectral sequence for regular schemes}, Ann. Sci. \'{E}cole Norm. Sup. (4)  {\bf 35} (2002), no. 1, 127-152.


\bibitem{BeDaGrMaZa24}
K. J. Becher, N. Daans, D. Grimm, G. Manzano-Flores and M. Zaninelli, {\em The Pythagoras number of a rational function field in two variables}, arXiv:2302.11425.

\bibitem{BeLe01}
G. Berhuy and D. B. Leep, {\em Divisible subgroups of Brauer groups and trace forms of central simple algebras}, Documenta Math. {\bf 6} (2001), 489--500.


\bibitem{ChRaRa13}  V. I. Chernousov, A. S. Rapinchuk,  I. A. Rapinchuk,
{\em The genus of a division algebra and the unramified Brauer group},
Bull. Math. Sci. {\bf 3} (2013), no. 2, 211--240.

\bibitem{ChRaRa15} V. I. Chernousov, A. S. Rapinchuk, I. A. Rapinchuk, {\em Division
algebras with the same maximal subfields},
Uspekhi Mat. Nauk {\bf 70} (2015), no. 1, 89--122 (in Russian). English translation:
Russian Math. Surveys {\bf 70} (2015), no. 1, 83--112.

\bibitem{ChRaRa16}
V. I. Chernousov, A. S. Rapinchuk, I. A. Rapinchuk, {\em On the size of the genus of a division algebra}, Proc. Steklov Inst. Math.  {\bf 292} (2016), no. 1, 63--93.

\bibitem{ChRaRa20}
V. I. Chernousov, A. S. Rapinchuk, I. A. Rapinchuk, {\em The finiteness of the genus of a finite-dimensional division algebra, and some generalizations}, Israel J. Math.  {\bf 236} (2020), no. 2, 747--799.

\bibitem{ChRaRa24}
V. I. Chernousov, A. S. Rapinchuk,  I. A. Rapinchuk, {\em Simple algebraic groups with the same maximal tori, weakly
commensurable Zariski-dense subgroups, and good reduction},  Adv. Math. {\bf 438} (2024), paper no. 109437.

\bibitem{CoSu98}
J.-L. Colliot-Th{\'e}l{\`e}ne and R. Sujatha, {\em Unramified Witt groups of real anisotropic quadrics}, Duke Math. J. {\bf 91} (1998), no. 2, 393--404.

\bibitem{ElKaMe08}
R. Elman, N. Karpenko and A. Merkurjev, {\em The algebraic and geometric theory of quadratic forms},
Amer. Math. Soc. Colloq. Publ., 56
American Mathematical Society, Providence, RI, 2008.

\bibitem{ElLa72}
R. Elman and T. Y. Lam, {\em Quadratic forms over formally real fields and pythagorean fields}, Amer. J. Math. {\bf 94} (1972), 1155--1194.

\bibitem{ElLa74}
R. Elman and T. Y. Lam, {\em Classification theorems for quadratic forms over fields}, Comment. Math. Helv. {\bf 49} (1974), 373--381.


\bibitem{Fe87}
E. Fern\'{a}ndez-Carmena, {\em On the injectivity of the map of the Witt group of a scheme into the Witt group of its function field}, Math. Ann. {\bf 277} (1987), no. 3, 453--468.

\bibitem{GaSa10}
S. Garibaldi, D. J. Saltman, {\em Quaternion algebras with the same subfields},
Quadratic forms, linear algebraic groups, and cohomology, 225--238, Dev. Math. 18,
Springer, New York, 2010.

\bibitem{GiSz06}
P. Gille and T. Szamuely, {\em Central simple algebras and Galois cohomology}, Cambridge Stud. Adv. Math., 101
Cambridge University Press, Cambridge, 2006.

\bibitem{Kn77}
M. Knebusch, {\em Symmetric bilinear forms over algebraic varieties}, Conference on Quadratic Forms - 1976 (Proc. Conf., Queen's Univ., Kingston, Ont., 1976), pp. 103--283
Queen's Papers in Pure and Appl. Math., No. 46
Queen's University, Kingston, ON, 1977.

\bibitem{KrLi24}
D. Krashen and M. Lieblich,
{\em Transcendental splitting fields of division algebras}, Amitsur Centennial Symposium, 191--202.
Contemp. Math., 800
Israel Math. Conf. Proc.
American Mathematical Society, [Providence], RI, 2024.

\bibitem{KrMaRaRoSa22}
D. Krashen, E. Matzri, A. S. Rapinchuk, L. H. Rowen and D. J. Saltman, {\em Division algebras with common subfields},
Manuscripta Math. {\bf 169} (2022), no. 1-2, 209--249.


\bibitem{LaLeTi93}
T. Y. Lam, D. B. Leep and J.-P. Tignol, {\em Biquaternion algebras and quartic extensions}, Inst. Hautes \'Etudes Sci. Publ. Math.  {\bf 77} (1993), 63--102.

\bibitem{Me14} J. S. Meyer, {\em Division algebras with infinite genus},
Bull. London Math. Soc. {\bf 46} (2014), no. 3, 463--468.

\bibitem{Pi82}
R. S. Pierce, {\em Associative algebras}, Stud. Hist. Modern Sci., 9
Grad. Texts in Math., 88
Springer-Verlag, New York-Berlin, 1982.

\bibitem{RaRa10}  A. S. Rapinchuk,  I. A. Rapinchuk,
{\em On division algebras having the same maximal subfields},
Manuscripta  Math. {\bf 132} (2010), no. 3-4, 273--293.

\bibitem{RaRa20}
A. S. Rapinchuk and I. A. Rapinchuk,  {\em Linear algebraic groups with good reduction}, Res. Math. Sci. {\bf 7} (2020), no. 3, paper no. 28.

\bibitem{Ra24}
I. A. Rapinchuk, {\em Algebraic groups with good reduction and the genus problem}, Amitsur Centennial Symposium, 271--281.
Contemp. Math., 800
Israel Math. Conf. Proc.
American Mathematical Society, [Providence], RI, 2024.

\bibitem{Sa99}
D. J. Saltman, {\em Lectures on Division Algebras}, CBMS Regional Conference Series in Mathematics, vol. 94, American Mathematical Society, Providence, RI, 1999.

\bibitem{Sr25}
S. Srimathy, {\em The genus of division algebras over discrete valued fields}, J. Th{\'e}or. Nombres Bordeaux {\bf 37} (2025), no. 2, 535--567.

\bibitem{Su90}
R. Sujatha, {\em Witt groups of real projective surfaces}, Math. Ann. {\bf 288} (1990), no. 1, 89--101.

\bibitem{Ti16}
S. V. Tikhonov, {\em Division algebras of prime degree with infinite genus},
Tr. Mat. Inst. Steklova {\bf 292} (2016),  264--267 (in Russian). English translation:
Proc. Steklov Inst. Math. {\bf 292} (2016), no. 1, 256--259.

\bibitem{Ti25}
S. V. Tikhonov, {\em Genus of division algebras over fields with infinite transcendence degree}, Arch. Math. (Basel) {\bf 125} (2025), no. 2, 115--121.



\end{thebibliography}
\end{document}